\documentclass[reqno]{amsart}
\usepackage[english]{babel}
\usepackage{graphicx}
\usepackage{color}
\usepackage[colorlinks=true,urlcolor=blue,linkcolor=blue,citecolor=blue]{hyperref}
\usepackage[capitalize]{cleveref}

\usepackage{amssymb}
\newcommand{\pattern}{family{}}
\newcommand{\monochromatic}{Ramsey}
\renewcommand{\patterns}{families{}}

\newcommand{\F}{\mathbb{F}}

\newcommand{\Q}{\mathbb{Q}}
\newcommand{\Z}{\mathbb{Z}}
\newcommand{\N}{\mathbb{N}}

\newcommand{\AR}{{\mathcal A}_R}
\newcommand{\AN}{{\mathcal A}_\mathbb{N}}

\newtheorem{theorem}{Theorem}[section]

\newtheorem{definition}[theorem]{Definition}
\newtheorem{example}[theorem]{Example}
\newtheorem{proposition}[theorem]{Proposition}
\newtheorem{question}[theorem]{Question}

\newtheorem{remark}[theorem]{Remark}
\newtheorem{lemma}[theorem]{Lemma}
\newtheorem{problem}[theorem]{Problem}
\newtheorem{corollary}[theorem]{Corollary}
\renewcommand{\emptyset}{\varnothing}
\theoremstyle{plain}
\newtheorem*{namedthm}{\namedthmname}
\newcounter{namedthm}
\makeatletter
	\newenvironment{named}[2]
	{\def\namedthmname{#1}
	\refstepcounter{namedthm}
	\namedthm[#2]\def\@currentlabel{#1}}
	{\endnamedthm}
\makeatother

\begin{document}

\author{Joel Moreira}
\email{moreira@math.ohio-state.edu}
\title{Monochromatic sums and products in $\N$}
\address{\small Department of Mathematics, The Ohio State University, Columbus, Ohio, USA}
\begin{abstract}An old question in Ramsey theory asks whether any finite coloring of the natural numbers admits a monochromatic pair $\{x+y,xy\}$.
We answer this question affirmatively in a strong sense by 
exhibiting a large new class of non-linear patterns which can be found in a single cell of any finite partition of $\N$.
Our proof involves a correspondence principle which transfers the problem into the language of topological dynamics.
As a corollary of our main theorem we obtain partition regularity for new types of equations, such as $x^2-y^2=z$ and $x^2+2y^2-3z^2=w$.
\end{abstract}
\maketitle

\section{Introduction}
In this paper we show that for any finite coloring (i.e. partition) of $\N=\{1,2,\dots\}$ there exist $x,y\in\N$ such that the set $\{x,x+y,xy\}$ is monochromatic.
In fact, we exhibit a rather large class of configurations with this property.
\subsection{Historical background and motivation}
A central topic in Ramsey theory is to understand which patterns can be found in one color of any finite coloring of the natural numbers.
We start with a definition:
\begin{definition}\label{def_ramsey}Let $k,s\in\N$, and let $f_1,\dots,f_k:\N^s\to\Z$.
We say that $\{f_1,\dots,f_k\}$ is a \emph{Ramsey \pattern{}} if for any finite coloring $\N=C_1\cup\cdots\cup C_r$, there exist ${\bf x}\in\N^s$ and $i\in\{1,\dots,r\}$ such that $\big\{f_1({\bf x}),\dots,f_k({\bf x})\big\}\subset C_i$.
\end{definition}

In this language, Schur's theorem \cite{Schur16} states that the \pattern{} $\{x,y,x+y\}$ \footnote{In a slight abuse of notation, we represent by $\{x,y,x+y\}$ the family comprised of the three functions $(x,y)\mapsto x$, $(x,y)\mapsto y$ and $(x,y)\mapsto x+y$.} is \monochromatic{} and van der Waerden's theorem \cite{vdWaerden27} states that for any $k\in\N$, the \pattern{} $\{x,x+y,\dots,x+(k-1)y\}$ is \monochromatic.
On the other hand, it is not hard to show that the \patterns{} $\{x,x+1\}$ and $\{x,y,3x-y\}$ are \emph{not} \monochromatic.
In 1933, Rado obtained a fundamental theorem describing necessary and sufficient conditions for a \pattern{} of linear functions to be \monochromatic{} \cite{Rado33}.
Inspired by Rado's result, we are led naturally to the following, by now classical, problem.
\begin{problem}\label{problem_polRado}
  Describe necessary and sufficient conditions on the polynomials $f_1,\dots,f_k\in\Z[x_1,\dots,x_s]$ that guarantee that the \pattern{} $\{f_1,\dots,f_k\}$ is \monochromatic.
\end{problem}

It follows from Schur's theorem that the \pattern{} $\{x,y,xy\}$ is \monochromatic{} (simply compose any given coloring $\chi:\N\to\{1,\dots,r\}$ with the map $n\mapsto 2^n$ to create a new coloring and apply Schur's theorem).
Using the same idea, van der Waerden's theorem implies that for each $k\in\N$ the \pattern{} $\{x,xy,\dots,xy^k\}$ is \monochromatic{}, and Rado's theorem implies that many more \patterns{} of the form $\{f_1,\dots,f_k\}$, where each $f_i$ is a monomial, are \monochromatic.

Configurations which combine both addition and multiplication, however, tend to be significantly harder to deal with: only in 1977 did Furstenberg and S\'ark\"ozy prove, independently, that the \pattern{} $\{x,x+y^2\}$ is monochromatic (cf. \cite[Theorem 1.2]{Furstenberg77} and \cite{Sarkozy78}), obtaining the first example of a non-linear \monochromatic{} \pattern{} which does not consist only of monomials.
Bergelson improved this result by showing that in fact the \pattern{} $\{x,y,x+y^2\}$ is \monochromatic{} \cite{Bergelson87b}.

The next major advance towards \cref{problem_polRado} was Bergelson and Leibman's polynomial extension of van der Waerden's theorem \cite{Bergelson_Leibman96} (see \cref{thm_polvdWoriginal} below).
In particular, they showed that for any polynomials $p_1,\dots,p_k\in\Z[x]$, the \pattern{} $\{x,x+p_1(y),\dots,x+p_k(y)\}$ is \monochromatic.
The polynomial van der Waerden theorem has now been extended in several directions (see, for instance, \cite{Bergelson_Furstenberg_McCutcheon96,Bergelson_Johnson_Moreira15,Bergelson_Leibman_Lesigne08}), each revealing new examples of polynomial \monochromatic{} \patterns.

In the last decade, many interesting polynomial Ramsey families were found \cite{Beiglbock_Bergelson_Hindman_Strauss06,Beiglbock_Bergelson_Hindman_Strauss08,Bergelson05,Frantzikinakis_Host14,McCutcheon10}, however a complete solution to \cref{problem_polRado} is still very far from reach.
In particular, the following simple question has remained unanswered for many years:

\begin{question}[cf. {\cite[Question 3]{Hindman_Imre_Strauss03}}, {\cite[Question 11]{Bergelson96}}]\label{question_mainproblem}
  Is the \pattern{} $\{x,y,x+y,xy\}$ \monochromatic?
\end{question}
This question was studied at least as early as 1979 by N. Hindman and R. Graham (see \cite[Section 4]{Hindman79} and \cite[pages 68-69]{Graham_Rothschild_Spencer80}), but even the \pattern{} $\{x+y,xy\}$ remained recalcitrant until now.
An affirmative answer to the analogue of \cref{question_mainproblem} in finite fields was recently obtained by Green and Sanders \cite{Green_Sanders15}, generalizing previous work by Shkredov \cite{Shkredov10} and Cilleruelo \cite{Cilleruelo12} (see also \cite{Vinh14} and \cite{Hanson13} for related results).

Bergelson and the author studied the analogue of \cref{question_mainproblem} for infinite fields in \cite{Bergelson_Moreira15,Bergelson_Moreira16}.
We showed, in particular, that the \pattern{} $\{x,x+y,xy\}$ is \monochromatic{} in any infinite field, and that for any finite coloring of $\Q$ there exist (many) $x\in\Q$ and $y\in\N$ such that $\{x+y,xy\}$ is monochromatic.
The methods of \cite{Bergelson_Moreira15} and \cite{Bergelson_Moreira16}, however, can not be directly used to establish that the family $\{x+y,xy\}$ is Ramsey in $\N$, the main problem being that the semigroup of affine transformations of $\N$ (which naturally appears in the dynamical approach to the problem) is not amenable.

\subsection{Main results}
The main result of this paper is the following:

\begin{theorem}\label{thm_mainN}
Let $s\in\N$ and, for each $i=1,\dots,s$, let $F_i$ be a finite set of functions $\N^i\to\Z$ such that for all $f\in F_i$ and any $x_1,\dots,x_{i-1}\in\N$, the function $x\mapsto f(x_1,\dots,x_{i-1},x)$ is polynomial with $0$ constant term.
Then for any finite coloring of $\N$ there exists a color $C\subset\N$ and (infinitely many) $(s+1)$-tuples $x_0,\dots,x_s\in\N$ such that 
$$\{x_0\cdots x_s\}\cup\Big\{x_0\cdots x_j+f(x_{j+1},\dots,x_i):0\leq j<i\leq s, f\in F_{i-j}\Big\}\subset C.$$
\end{theorem}

In particular, taking $s=1$ and $F_1=\{x\mapsto0,x\mapsto x\}$ consisting only of the zero function and the identity function, we deduce
\begin{corollary}\label{cor_{x,xy,x+y}}
For any finite coloring of $\N$ there exist (infinitely many) $x,y\in\N$ such that
$\{x,xy,x+y\}$ is monochromatic.
\end{corollary}

As an illustration, setting $s=5$ in \cref{thm_mainN} and letting each $F_i$ consist only of the function $f_i:(x_1,\dots,x_i)\mapsto x_1\cdots x_i$, we obtain the following (aesthetically pleasing) \monochromatic{} \pattern{}.
\begin{example}
  The following \pattern{} is \monochromatic:
  $$\left\{\begin{array}{ccccc}x&&&&\\xy,&x+y&&&\\xyz,&x+yz,&xy+z&&\\ xyzt,&x+yzt,&xy+zt,&xyz+t&\\ xyztw,&x+yztw,&xy+ztw,&xyz+tw&xyzt+w\end{array}\right\}$$
\end{example}

\cref{thm_mainN} can also be used to obtain new partition regular equations:

\begin{corollary}\label{cor_partitionregularequationgeneral}
  Let $k\in\N$ and $c_1,\dots,c_k\in\Z\setminus\{0\}$ be such that $c_1+\cdots+c_k=0$.
  Then for any finite coloring of $\N$ there exist pairwise distinct $a_0,\dots,a_k\in\N$, all of the same color, such that
  $$  c_1a_1^2+\cdots+c_ka_k^2=a_0.$$
\end{corollary}

In particular, setting $k=2$ and $c_1=1$, $c_2=-1$, we deduce:
\begin{corollary}\label{cor_partitionregularequation}
   For any finite coloring of $\N$ there exists a solution $a,b,c$ of the equation $a^2-b^2=c$ with all $a,b$ and $c$ of the same color.
\end{corollary}
Note that the similar equation $a^2-b=c$ is \emph{not} partition regular (cf. \cite[Theorem 3]{Csikvari_Gyarmati_Sarkozy12}).
\cref{cor_partitionregularequationgeneral} is proved in \cref{sec_corollaries}.

Our proof of \cref{thm_mainN} proceeds by first transferring the problem to the language of topological dynamics using a \emph{correspondence principle} (\cref{thm_correspondence}), then solving the dynamical problem using ideas developed in \cite{Bergelson_Moreira15} together with a ``complexity reduction'' method inspired by \cite{Bergelson_Leibman96}.
The correspondence principle is of independent interest because it allows one to formulte in dynamical terms the question of whether general polynomial \patterns{} are \monochromatic{}; we postpone the precise statement to \cref{sec_correspondence} because it uses notation and terminology from \cref{subsec_affinesemigroup}. 

The proof of \cref{thm_mainN} can be made elementary; to illustrate this, we present in \cref{sec_combinatorics} a short and purely combinatorial proof of \cref{cor_{x,xy,x+y}} which is independent from the rest of the paper. 
This combinatorial version of the proof is shorter but less transparent, avoiding the correspondence principle but consequentially obscuring the theorem's dynamical underpinnings.

The paper is organized as follows:
In \cref{sec_notation} we introduce some notation and establish some conventions to be used in the paper.
In \cref{sec_correspondence} we state and prove the correspondence principle, thereby reducing \cref{thm_mainN} to a statement in topological dynamics, \cref{thm_maindynamical}, which is proved in \cref{sec_proofmain}.
In \cref{sec_combinatorics} we present a more direct and combinatorial rendering of our dynamical proof of \cref{cor_{x,xy,x+y}}.
In \cref{sec_corollaries} we explore some combinatorial corollaries of our main result.
Finally, \cref{sec_finalremarks} is devoted to an extension of our results to a general class of rings.

\

\paragraph{\textbf{Acknowledgements}}
The author thanks Marc Carnovale, Daniel Glasscock, Andreas Koutsogiannis and Pedro Vieira for helpful comments on an earlier version of the paper. Thanks are also due to Donald Robertson and Florian Richter for insightful discussions which planted the seed for some of the main ideas in this paper. Special thanks go to Vitaly Bergelson for all of the above and for his constant support and encouragement.

\section{Definitions, notation and conventions}\label{sec_notation}
\subsection{The affine semigroup}\label{subsec_affinesemigroup}
We denote by $\AN^-$ the semigroup consisting of all the maps $x\mapsto ax+b$ from $\Z$ to itself, where $a\in\N$ and $b\in\Z$, and with composition of functions as the semigroup operation.
For a given $u\in\Z$, the map $x\mapsto x+u$ is denoted by $A_u$ and, if $u>0$, the map $x\mapsto ux$ is denoted by $M_u$.
The distributivity law can be written as
\begin{equation}\label{eq_distributivity}
\forall u\in\N,\ v\in\Z\qquad M_uA_v=A_{uv}M_u.
\end{equation}
Given an action $(T_g)_{g\in\AN^-}$ of $\AN^-$ on a set $X$ (meaning that for each $g\in\AN^-$, there is a map $T_g:X\to X$ and for any $g,h\in\AN^-$ we have the composition law $T_g\circ T_h=T_{gh}$) and $u\in\Z$, we will frequently denote, abusing notation slightly, the map $T_{A_u}$ simply by $A_u$ and, if $u>0$, the map $T_{M_u}$ by $M_u$.

Given a semigroup $G$, a $G$-\emph{topological system} is a pair $(X,(T_g)_{g\in G})$ where $X$ is a compact Hausdorff space (not necessarily metrizable) and $(T_g)_{g\in G}$ is an action by continuous functions $T_g:X\to X$.
A system $(X,(T_g)_{g\in G})$ is \emph{minimal} if $X$ contains no proper non-empty closed invariant subsets.
A point $x\in X$ is a \emph{minimal point} if its orbit closure $Y:=\overline{\{T_gx:g\in G\}}$ is a minimal subsystem of $X$ (i.e., if $(Y,(T_g|_Y)_{g\in G})$ is a minimal system).

Observe that any $\AN^-$-topological system $(X,(T_g)_{g\in\AN^-})$ naturally induces a $(\Z,+)$-topological system $(X,(S_u)_{u\in\Z})$, by letting $S_u:=T_{A_u}$.
A point $x\in X$ is called \emph{additively minimal} if it is a minimal point for the system $(X,(S_u)_{u\in\Z})$.

\subsection{Piecewise syndetic sets}\label{subsec_pws}
Given sets $E,H\subset\N$ and a number $n\in\N$ we use the following notation:
\begin{itemize}
  \item $nE:=\{nm:m\in E\}$,
  \item $n+E:=\{n+m:m\in E\}$,
  \item $E-n:=\{m-n:m\in E, m>n\}=\{x\in\N:x+n\in E\}$,
  \item $E+H:=\{m+n:m\in E, n\in H\}$,
  \item $E-H:=\{m-n:m\in E,n\in H, m>n\}=\bigcup_{n\in H}E-n$.
\end{itemize}

A subset $S\subset\N$ is called \emph{syndetic} if it has bounded gaps.
More precisely, $S$ is syndetic if there exists a finite set $F\subset\N$ such that $\N=S-F$.
A set $T\subset\N$ is called \emph{thick} if it contains arbitrarily long intervals or, equivalently, if it has non-empty intersection with every syndetic set.
A set $E\subset\N$ is called \emph{piecewise syndetic} if it is the intersection of a syndetic set and a thick set.

If $E$ is a piecewise syndetic set and $E\subset H$, then $H$ is also a piecewise syndetic set.
Observe that for any piecewise syndetic set $E\subset\N$ and any $n\in\N$, the sets $nE$, $n+E$ and $E-n$ are all piecewise syndetic.
Furthermore, if any of $nE$, $n+E$ or $E-n$ are piecewise syndetic, then so is $E$.
Therefore we have:
\begin{proposition}\label{prop_pwsaffine}
  Let $E\subset\N$ and $g\in\AN^-$.
   The set $E$ is piecewise syndetic if and only if its image $g(E)$ also is.
\end{proposition}
We will also make use of the following well known property of piecewise syndetic sets.
\begin{proposition}[see, for instance, {\cite[Theorem 1.24]{Furstenberg81}}]\label{prop_pwspartition}
  Let $E\subset\N$ be a piecewise syndetic set. Then for any finite partition of $E=E_1\cup\cdots\cup E_r$, one of the pieces $E_t$ is piecewise syndetic.
\end{proposition}

\section{An affine topological Correspondence Principle}\label{sec_correspondence}

In this section we reduce \cref{thm_mainN} to the following statement in topological dynamics:
\begin{theorem}\label{thm_maindynamical}
Let $(X,(T_g)_{g\in\AN^-})$ be an $\AN^-$-topological system with a dense set of additively minimal points, and assume that each map $T_g:X\to X$ is open and injective.
Let $s\in\N$ and, for each $i=1,\dots,s$, let $F_i$ be a finite set of functions $\N^i\to\Z$ such that for all $f\in F_i$ and any $x_1,\dots,x_{i-1}\in\N$, the function $x\mapsto f(x_1,\dots,x_{i-1},x)$ is polynomial with $0$ constant term.
Then for any open cover ${\mathcal U}$ of $X$ there exists an open set $U\in{\mathcal U}$ in that cover and infinitely many $s$-tuples $x_1,\dots,x_s\in\N$ such that
$$U\cap \bigcap_{0\leq j<i\leq s}\bigcap_{f\in F_{i-j}}M_{x_{j+1}\cdots x_s} A_{f(x_{j+1},\dots,x_i)}U\neq\emptyset$$
\end{theorem}
The proof of \cref{thm_maindynamical} is presented in \cref{sec_proofmain}.

\subsection{Reducing \cref{thm_mainN} to \cref{thm_maindynamical}}
The elegant idea of using topological dynamics to find \monochromatic{} \patterns{} on $\N$ was developed by Furstenberg and Weiss in \cite{Furstenberg_Weiss78}.
They considered each coloring $\chi:\N\to\{1,\dots,r\}$ as a point in the symbolic system $(\{1,\dots,r\}^\N,T)$ (where $T$ is the left shift), and observed that it is possible to reformulate van der Waerden's theorem (among many others) as a multiple recurrence result on minimal subsystems of $(\{1,\dots,r\}^\N,T)$.
By proving the resulting multiple recurrence theorem (\cite[Theorem 1.5]{Furstenberg_Weiss78}), they obtained a new proof of van der Waerden's theorem (and indeed of it's multidimensional version, due originally to Tibor Gr\"unwald).
This correspondence is now a standard technique; for instance it was used by Bergelson and Leibman in their proof of the polynomial van der Waerden's theorem \cite[Corollary 1.11]{Bergelson_Leibman96} (see \cref{thm_polvdW}).

Unfortunately, the same procedure does not allow one to deduce \cref{thm_mainN} from \cref{thm_maindynamical}.
This is essentially because the configurations in \cref{thm_mainN} are not invariant under shifts (additive or multiplicative): if $P$ is a set of the form $\{xy,x+y\}$ and $c\in\N$, then in general neither $P+c$ nor $Pc$ is of the same form.
By contrast, observe that arithmetic progressions are invariant under \emph{both} addition and multiplication, in the sense that for any arithmetic progression $P$ and any $c\in\N$, both $P+c$ and $Pc$ are arithmetic progressions of the same length.

Nevertheless we obtained the following correspondence principle.

\begin{theorem}\label{thm_correspondence}
  There exists an $\AN^-$-topological system $(X,(T_g)_{g\in\AN^-})$ with a dense set of additively minimal points, such that each map $T_g:X\to X$ is open and injective, and with the property that for any finite coloring $\N=C_1\cup\cdots\cup C_r$ there exists an open cover $X=U_1\cup\cdots\cup U_r$ such that for any $g_1,\dots,g_k\in\AN^-$ and $t\in\{1,\dots,r\}$,
  \begin{equation}\label{eq_thm_correspondence}
    \bigcap_{\ell=1}^kT_{g_\ell}(U_t)\neq\emptyset\qquad\Longrightarrow\qquad\N\cap\bigcap_{\ell=1}^kg_\ell(C_t)\neq\emptyset
  \end{equation}
\end{theorem}

\begin{remark}\label{rmkr_pws}
It follows from the proof of \cref{thm_correspondence} that the system $(X,(T_g)_{g\in\AN^-})$ also has the property that for any piecewise syndetic set $C_t\subset\N$ there exists a non-empty open set $U_t\subset X$ such that \eqref{eq_thm_correspondence} holds for any $g_1,\dots,g_k\in\AN^-$.
\end{remark}

\begin{remark}\label{rmrk_pws2}
It follows from the proof of \cref{thm_correspondence} that the intersection $\N\cap\bigcap_{j=1}^kg_j(C_t)$ (both in the theorem and in \cref{rmkr_pws}) is not only non-empty but is in fact piecewise syndetic.
\end{remark}

We can now derive \cref{thm_mainN} from its topological counterpart \cref{thm_maindynamical} and the correspondence principle \cref{thm_correspondence}.

\begin{proof}[Proof of \cref{thm_mainN}]
Let $s\in\N$ and, for each $i=1,\dots,s$, let $F_i$ be a finite set of functions $\N^i\to\Z$ such that for all $f\in F_i$ and any $x_1,\dots,x_{i-1}\in\N$, the function $x\mapsto f(x_1,\dots,x_{i-1},x)$ is polynomial with $0$ constant term.
Let $\N=C_1\cup\cdots\cup C_r$ be a finite coloring of $\N$.
We need to show that there exists a color $C_t$ and (infinitely many) $s+1$-tuples $x_0,\dots,x_s\in\N$ such that $x_0\cdots x_s\in C_t$ and, for every $0\leq j<i\leq s$ and $f\in F_{i-j}$, we have $x_1\cdots x_j+f(x_{j+1},\dots,x_i)\in C_t$.

We append to $F_s$ the zero function $f:\N^s\to\{0\}$ if necessary.
Invoking \cref{thm_correspondence} and then \cref{thm_maindynamical}, we find a color $C_t$ and (infinitely many) $s$-tuples $x_1,\dots,x_s\in\N$ such that the intersection
\begin{equation}\label{eq_proof_main1}\N\cap C_t\cap  \bigcap_{0\leq j<i\leq s}\bigcap_{f\in F_{i-j}}M_{x_{j+1}\cdots x_s} A_{-f(x_{j+1},\dots,x_i)}C_t
\end{equation}
is non-empty.
Take $x$ in the intersection \eqref{eq_proof_main1} and observe that $x\in x_1\cdots x_sC_t$ (letting $j=0$, $i=s$ and $f\equiv0$).
Therefore $x_0:=x/(x_1\cdots x_s)\in C_t$ (and in particular is an integer).

Finally, for $0\leq j<i\leq s$ and $f\in F_{i-j}$, we have $x\in x_{j+1}\cdots x_s\big(C_t-f(x_{j+1},\dots,x_i)\big)$, so $x_0\cdots x_j+f(x_{j+1},\dots,x_i)=x/(x_{j+1}\cdots x_s)+f(x_{j+1},\dots,x_i)\in C_t$.
\end{proof}

\subsection{Proof of the correspondence principle}
The remainder of this section is dedicated to the proof of \cref{thm_correspondence}.
The construction of $X$ is quite explicit as a subset of the Stone-\v Cech compactification of $\N$, realized as the space of ultrafilters on $\N$.
In this setting, the action of $\AN^-$ on $X$ is natural.
The idea of using the Stone-\v Cech compactification to prove the correspondence principle was inspired by its implicit use in \cite{Beiglbock11} (in the setting of measurable dynamics).
We start by summarizing some facts about ultrafilters which we will use, refering the reader to \cite[Section 3]{Bergelson96} for a short and friendly introduction on the subject, and to \cite{Hindman_Strauss98} for a complete treatment.
We will only make use of the facts and definitions about ultrafilters in this section. 

An \emph{ultrafilter} on $\N$ is a non-empty family $p$ of subsets of $\N$ which is closed under intersections and supersets, and which satisfies the property $E\in p\iff (\N\setminus E)\notin p$.
For each $x\in\N$, the family $p_x=\{E\subset\N:x\in E\}$ is an ultrafilter;
ultrafilters of this form are called \emph{principle}.
The existence of non-principle ultrafilters requires (at least some weak form of) the axiom of choice.

Denote by $\beta\N$ the set of all ultrafilters over $\N$.
The sets of the form $\overline E:=\{p\in\beta\N:E\in p\}$ with $E\subset\N$ form a base for a topology on $\beta\N$.
With this topology $\beta\N$ becomes a compact Hausdorff space (cf. \cite[Theorem 2.18]{Hindman_Strauss98}) and can be identified with the Stone-\v Cech compactification of $\N$ (cf. \cite[Theorem 3.27]{Hindman_Strauss98}), where $\N$ is embedded densely inside $\beta\N$ by identifying each $x\in\N$ with the corresponding principal ultrafilter $p_x$.

There is a natural action $(T_g)_{g\in\AN^-}$ of $\AN^-$ on the set $\beta\N\setminus\N$ of non-principle ultrafilters, described as follows.
For $g\in\AN^-$, the map $T_g:\beta\N\setminus\N\to\beta\N\setminus\N$ takes $p\in\beta\N\setminus\N$ to
\begin{equation}\label{eq_defineTg}
  T_g(p):=\big\{E\subset\N:g^{-1}(E)\in p\big\}=\big\{E\subset\N:\{x\in\N:g(x)\in E\}\in p\big\}
\end{equation}
\begin{remark}
  An equivalent way to define $T_g$ is to start with a map $T_g:\beta\N\to\beta\Z$, defined on principal ultrafilters via the formula $T_g(p_x)=p_{g(x)}$ and then extend it to $\beta\N$ using the universal property of the Stone-\v Cech compactification.
One can then check that for a non-principle ultrafilter $p\in\beta\N\setminus\N$, the image $T_g(p)$ is in fact in $\beta\N\setminus\N$ and corresponds to the ultrafilter described in \eqref{eq_defineTg}.
We will not make use of this fact.
\end{remark}

\begin{lemma}\label{lemma_ultrafilterswork}
  For each $g\in\AN^-$, the map $T_g:\beta\N\setminus\N\to\beta\N\setminus\N$ is continuous, open and injective.
  Moreover, for $g,h\in\AN^-$ one has $T_g\circ T_h=T_{gh}$.
\end{lemma}

\begin{proof}
  One can easily check (using only the definitions) that $T_g(p)$ is indeed a non-principle ultrafilter and that $T_g\circ T_h=T_{gh}$.
  To show that $T_g$ is continuous, take an open set $\overline E\subset\beta\N$ for $E\subset\N$ infinite; we need to show that $T_g^{-1}(\overline E)$ is open.
  We have
  $$p\in T_g^{-1}(\overline E)\iff E\in T_g(p)\iff g^{-1}(E)\in p$$
  therefore $T_g^{-1}(\overline E)=\overline{g^{-1}(E)}$ is open and $T_g$ is continuous.

  To show that $T_g$ is injective, let $p\neq q$ be in $\beta\N\setminus\N$ and let $E\in p\setminus q$.
  Since $g:\N\to\Z$ is injective we have that $g^{-1}(g(E)\cap\N)$ is a subset of $E$; since $E\notin q$, it follows that also $g^{-1}(g(E)\cap\N)\notin q$ and hence $g(E)\cap\N\notin T_g(q)$.
  On the other hand, $g(E)\cap\N$ is a co-finite subset of $g(E)$, which implies that $g^{-1}(g(E)\cap\N)$ is a co-finite subset of $E$.
  Since $p$ is non-principal, it can not contain finite sets, therefore $g^{-1}(g(E)\cap\N)\in p$ and hence $g(E)\cap\N\in T_g(p)$.
  This shows that $T_g(p)\neq T_g(q)$, proving injectivity.

  Finally we show that $T_g$ is open. Let $E\subset\N$ be infinite; we will show that $T_g(\overline E\setminus\N)=\overline{g(E)\cap\N}\setminus\N$, which will imply that $T_g:\beta\N\setminus\N\to\beta\N\setminus\N$ is indeed open.
  As in the proof of injectivity, if $p\in\overline E$ is non-principal, then $g(E)\cap\N\in T_g(p)$, proving one of the inclusions.
  Conversely, if $p\in\beta\N\setminus\N$ is such that $g(E)\cap\N\in T_g(p)$, then $g^{-1}(g(E)\cap\N)\in p$ and hence $E\in p$, proving the other inclusion and finishing the proof.
\end{proof}

 \cref{lemma_ultrafilterswork} implies that $(T_g)_{g\in\AN^-}$ is an action on $\beta\N\setminus\N$ and hence $\big(\beta\N\setminus\N,(T_g)_{g\in\AN^-}\big)$ is an $\AN^-$-topological dynamical system.
We are now ready to prove \cref{thm_correspondence}.

\begin{proof}[Proof of \cref{thm_correspondence}]
Let $Y\subset\beta\N\setminus\N$ be the set of all additively minimal points in $\big(\beta\N\setminus\N,(T_g)_{g\in\AN^-}\big)$ and let $X:=\overline{Y}$ be its closure.
It is usual to denote $Y=K(\beta\N,+)$.
We will show that for each $g\in\AN^-$, $T_g$ maps $X$ into $X$.

It follows from \cite[Corollary 4.41]{Hindman_Strauss98} that an ultrafilter $p\in\beta\N$ is in $X=\overline{K(\beta\N,+)}$ if and only if every member $E\in p$ is piecewise syndetic.
Take $p\in X$ and $g\in\AN^-$; we claim that $T_g(p)\in X$.
Using the definition, it suffices to show that if $g^{-1}(E)$ is piecewise syndetic, then so is $E$.
It follows from \cref{prop_pwsaffine} that if $g^{-1}(E)$ is piecewise syndetic, then so is $g(g^{-1}(E))$, and since $g(g^{-1}(E))\subset E$ we conclude that $E$ is also piecewise syndetic.
This shows that each $g\in\AN^-$ induces a natural continuous map $T_g:X\to X$.
Moreover, a similar argument shows that if $p\in\beta\N\setminus\N$ and $g\in\AN^-$ are such that $T_g(p)\in X$, then $p\in X$; therefore $T_p:X\to X$ is also open.

So far we constructed a compact Hausdorff space $X$ together with an action $(T_g)_{g\in\AN^-}$ of $\AN^-$ on $X$ by continuous injective open maps with a dense set of additively minimal points.
To finish the proof, consider a coloring $\N=C_1\cup\cdots\cup C_r$ and let $U_t:=\{p\in X:C_t\in p\}=\overline{C_t}\cap X$ for each $t\in\{1,\dots,r\}$.
Then each $U_t$ is a (possibly empty) open subset of $X$ and each $p\in X$ belongs to some $U_t$.
Now let $g_1,\dots,g_k\in\AN^-$ and $t\in\{1,\dots,r\}$ be such that $\bigcap_{\ell=1}^k T_{g_\ell}(U_t)\neq\emptyset$.
Then, since the maps $T_{g_\ell}:X\to X$ are open, it follows that $\bigcap_{\ell=1}^k T_{g_\ell}(U_t)$ is a non-empty open subset of $X$.
Take any $p$ in this intersection; we claim that $g_\ell(C_t)\cap\N\in p$ for any $\ell\in\{1,\dots,k\}$.

Indeed, for each $\ell\in\{1,\dots,k\}$, there exists $p_\ell\in U_t\subset\overline{C_t}$ such that $p=T_{g_\ell}(p_\ell)$.
Since $g_\ell^{-1}(g_\ell(C_t)\cap\N)$ is a co-finite subset of $C_t$ and $p_\ell$ is non-principal, it follows that $g_\ell^{-1}(g_\ell(C_t)\cap\N)\in p_\ell$ and hence indeed $g_\ell(C_t)\cap\N\in p$, as desired.
Finally, it follows that the finite intersection $\N\cap\bigcap_{\ell=1}^kg_\ell(C_t)$ is also in $p$ and hence is non-empty.
\end{proof}

\section{Proof of \cref{thm_maindynamical}}\label{sec_proofmain}
\subsection{A version of the polynomial van der Waerden theorem}
We will make use of the polynomial van der Waerden theorem of Bergelson and Leibman:
\begin{theorem}[cf. {\cite[Corollary 1.11]{Bergelson_Leibman96}}]\label{thm_polvdWoriginal}
  Let $F\subset\Z[x]$ be a finite set of polynomials such that $p(0)=0$ for all $p\in F$.
  Then for any finite coloring of $\N$ there exist $x,y\in\N$ such that the set $\{x+p(y):p\in F\}$ is monochromatic.
\end{theorem}

As mentioned in the previous section, the proof of \cref{thm_polvdWoriginal} in \cite{Bergelson_Leibman96} is derived from a topological statement.
While this topological statement (namely, \cite[Theorem C]{Bergelson_Leibman96}) is only proved for metrizable spaces, it is remarked in \cite[Proposition 1.10]{Bergelson_Leibman96} that the result holds in the non-metrizable setting, either by running a similar proof or by applying the combinatorial version of polynomial van der Waerden directly.
We use the second approach to derive the following corollary, which is a dynamical version of \cref{thm_polvdWoriginal} in the form we will use.
\begin{corollary}
\label{thm_polvdW}
Let $(X,(T_g)_{g\in\AN^-})$ be an $\AN^-$-topological dynamical system, and assume that $X$ contains a dense set of additively minimal points.
Let $F\subset\Q[x]$ be a finite set such that $p(0)=0$ for all $p\in F$.
Then for any nonempty open set $U\subset X$ there exists $n\in\N$ such that $p(n)\in\Z$ for each $p\in F$ and
$$\bigcap_{p\in F}A_{p(n)}U\neq\emptyset$$
\end{corollary}
\begin{proof}
  Let $y\in U$ be an additively minimal point, and let $Y=\overline{\{A_ny:n\in\Z\}}$ be its additive orbit closure. Since $\big(Y,(A_n)_{n\in\Z}\big)$ is a minimal topological system, the union $\bigcup_n A_nU$ covers $Y$, and by compactness there exists $r\in\N$ for which the finite union $\bigcup_{n=1}^r A_nU$ covers $Y$.
  We define a coloring $\chi:\N\to\{1,\dots,r\}$ of $\N$ by letting $\chi(n)$ be such that $A_ny\in A_{\chi(n)}U$.

  Let $m\in\N$ be a common multiple of the denominators of the coefficients of every $p\in F$.
  For each polynomial $p\in F$, let $\tilde p:n\mapsto-p(mn)$ and observe that $\tilde p\in\Z[x]$ and $\tilde p(0)=0$.
  We invoke \cref{thm_polvdWoriginal} with $\tilde F=\{\tilde p:p\in F\}$ to find some $t\in\{1,\dots,r\}$ and $x,z\in\N$ such that $\chi\big(x+\tilde p(z)\big)=t$ for every $p\in F$.
  In other words, $A_{x-p(mz)}y\in A_tU$ for all $p\in F$ and hence, letting $n=mz$, we deduce that $A_{x-t}y\in A_{p(n)U}$ for every $p\in F$.
  We conclude that
  $$A_{x-t}y\quad\in \quad \bigcap_{p\in F}A_{p(n)}U,$$
  proving the intersection to be non-empty.
\end{proof}

\subsection{Outline of the proof}\label{sec_motivation}
There are two main ingredients in the proof of \cref{thm_maindynamical}.
One is a ``complexity reduction'' technique similar to the one used by Bergelson and Leibman in \cite{Bergelson_Leibman96} to prove the polynomial van der Waerden theorem (and also used in \cite[Lemma 8.5]{Bergelson_Moreira16b}).
The other main ingredient is a fact about the algebraic behaviour of the expression $g:n\mapsto M_nA_{f(n)}\in\AN^-$ discovered (and explored) in \cite{Bergelson_Moreira15}, namely that the ``multiplicative derivative'' $n\mapsto g(nm)g(n)^{-1}$ becomes a purely additive expression whenever $f$ is a polynomial.

Before we delve into the full details of the proof of \cref{thm_maindynamical} in the next subsection, we explain the main steps of the proof in the special case when $s=1$ and $F_1$ is a singleton consisting only of the map $x\mapsto-x$.
In other words, we will show that for any finite cover of a nice $\AN^-$-topological system $X$, there is a set $U$ in the cover and some $y\in\N$ such that $U\cap M_yA_{-y}U\neq\emptyset$ (after applying the correspondence principle this special case corresponds essentially to \cref{cor_{x,xy,x+y}}).

The idea is to construct a sequence $(B_n)$ of non-empty open sets of $X$, each contained inside some member $U_n$ of the open cover, such that
\begin{equation}\label{eq_outline}
  \forall\ n<m, \quad\exists\ y=y(n,m)\in\N,\qquad\quad M_yA_{-y}B_n\supset B_m.
\end{equation}
Assuming we construct such sequence, since the open cover is finite we can find $n<m$ for which both $B_n$ and $B_m$ are contained inside the same member $U$ of the open cover; it then follows from \eqref{eq_outline} that $U\cap M_yA_{-y}U\neq\emptyset$, finishing the proof.

The construction of the sequence $(B_n)$ is natural and is illustrated by Figure \ref{figure}: starting with an arbitrary non-empty open set $B_0$, we find some $y_1$ such that $B_0\cap A_{-y_1}B_0\neq\emptyset$ (such $y_1$ exists since $B_0$ contains some additively minimal points), and then we ``push'' that intersection by $M_{y_1}$ to create $B_1:=M_{y_1}(B_0\cap A_{-y_1}B_0)$.
In particular, \eqref{eq_outline} holds for $n=0,m=1$ with $y=y_1$.

For the next step, we start similarly: assume $y_2\in\N$ is such that $B_1\cap A_{-y_2}B_1\neq\emptyset$. 
As long as we take $B_2\subset M_{y_2}(B_1\cap A_{-y_2}B_1)$, we will indeed have $B_2\subset M_{y_2}A_{-y_2}B_1$ (and hence \eqref{eq_outline} holds for $n=1$ and $m=2$).
Next we need to force $B_2$ to satisfy \eqref{eq_outline} for $n=0$ and $m=2$.
Since we know how to control the ``multiplicative derivative'' of the expression $M_yA_{-y}$, we seek to obtain \eqref{eq_outline} with $y(0,2)=y_1y_2$; in other words, we want $B_2\subset M_{y_1y_2}A_{-y_1y_2}B_0$.
Putting both conditions together, we are left to find $y_2\in\N$ so that
$$M_{y_2}(B_1\cap A_{-y_2}B_1)\cap M_{y_1y_2}A_{-y_1y_2}B_0\neq\emptyset.$$
Applying $M_{y_2}^{-1}$ it suffices to make $B_1\cap A_{-y_2}B_1\cap M_{y_1}A_{-y_1y_2}B_0\neq\emptyset$.
Using the distributivity law \eqref{eq_distributivity}, we have that $M_{y_1}A_{-y_1y_2}=A_{-y_1^2y_2}M_{y_1}$, and since $M_{y_1}B_0\supset M_1$, we see that it is sufficient to find $y_2\in\N$ such that
$$B_1\cap A_{-y_2}B_1\cap A_{-y_1^2y_2}B_1\neq\emptyset.$$

The existence of such a $y_2$ is a consequence of \cref{thm_polvdW}, so setting $B_2:=M_{y_2}(B_1\cap A_{-y_2}B_1\cap A_{-y_1^2y_2}B_1)$ we have successfully constructed $B_2$ and $y_2$ satisfying \eqref{eq_outline} whenever $n\leq2$.

Proceeding in this fashion we can construct the sequence $B_n$, each time invoking \cref{thm_polvdW} to choose $y_n\in\N$ so that
$$B_n:=M_{y_n}(B_{n-1}\cap A_{-y_n}B_{n-1}\cap A_{-y_{n-1}^2y_n}B_{n-1}\cap\cdots\cap A_{-y_1^2\cdots y_{n-1}^2y_n}B_{n-1})$$
is non-empty.
One can see, using the distributivity law \eqref{eq_distributivity}, that \eqref{eq_outline} indeed holds with $y(n,m)=y_{n+1}\cdots y_m$.
For instance, to see why $M_{y_2y_3y_4}A_{-y_2y_3y_4}B_1\supset B_4$, observe that
$$M_{y_2y_3y_4}A_{-y_2y_3y_4}B_1=M_{y_4}A_{-y_2^2y_3^2y_4}M_{y_3}M_{y_2}B_1\subset M_{y_4}A_{-y_2^2y_3^2y_4}B_3\subset B_4.$$
\begin{figure}
  \centering
  \includegraphics[width=10cm]{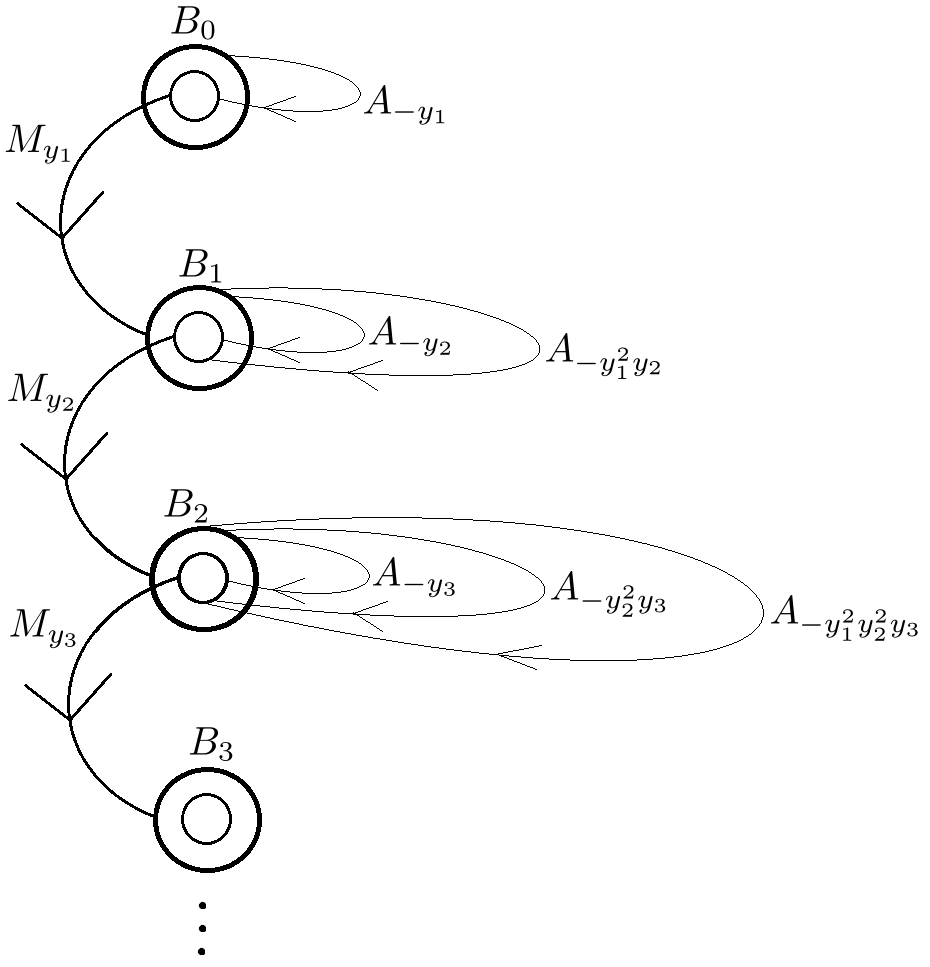}
  \caption{Construction of the sequence $(B_n)$}
  \label{figure}
\end{figure}

\subsection{Proof of \cref{thm_maindynamical}}
Let $(X,(T_g)_{g\in\AN^-})$ be an $\AN^-$-topological system with a dense set of additively minimal points and assume that each map $T_g:X\to X$ is open and injective.
Let $s\in\N$ and, for each $i=1,\dots,s$, let $F_i$ be a finite set of functions $\N^i\to\Z$ such that for all $f\in F_i$ and any $x_1,\dots,x_{i-1}\in\N$, the function $x\mapsto f(x_1,\dots,x_{i-1},x)$ is polynomial with $0$ constant term.
Let ${\mathcal U}$ be an open cover of $X$.
We need to find $U\in{\mathcal U}$ and infinitely many $s$-tuples $x_1,\dots,x_s\in\N$ such that
\begin{equation}\label{eq_proof_thm_maindynamicnovo}
  U\cap \bigcap_{0\leq j<i\leq s}\bigcap_{f\in F_{i-j}}M_{x_{j+1}\cdots x_s} A_{f(x_{j+1},\dots,x_i)}U\neq\emptyset.
\end{equation}
Since $X$ is compact, we can find a finite subcover $U_1,\dots,U_r$ of ${\mathcal U}$ with each $U_t\neq\emptyset$.

  We will construct, inductively, four sequences:
   \begin{itemize}
     \item $(t_n)_{n\geq0}$ in $\{1,\dots,r\}$,
     \item $(y_n)_{n\geq1}$ in $\N$ increasing,
     \item $(B_n)_{n\geq0}$ of non-empty open subsets of $X$,
     \item $(D_n)_{n\geq1}$ of non-empty open subsets of $X$,
   \end{itemize}
    such that $B_n\subset U_{t_n}$ (the set $D_n$ corresponds to the smaller circle inside $B_{n-1}$ in Figure \ref{figure}).
    It will be convenient to denote by $y(m,n)\in\N$ the product $y(m,n):=y_{m+1}y_{m+2}\cdots y_n$ for any $0\leq m\leq n$, with the convention that the (empty) product $y(n,n)$ equals $1$.
  
  Initiate $t_0=1$ and $B_0=U_1$.
  Using \cref{thm_polvdW} we find $y_1\in\N$ such that
  $$D_1:=B_0\cap\bigcap_{f\in F_1}A_{f(y_1)}B_0\neq\emptyset.$$
  Since $U_1,\dots,U_r$ forms an open cover of $X$ and $M_n:X\to X$ is an open map, we can find $t_1\in\{1,\dots,r\}$ such that $B_1:=M_{y_1}D_1\cap U_{t_1}$ is open and nonempty.
  Next we invoke \cref{thm_polvdW} again to find $y_2\in\N$ such that
  $$D_2:=B_1\cap\left(\bigcap_{f\in F_1}A_{f(y_2)}B_1\cap A_{y_1f(y_1y_2)}B_1\right)\cap\left(\bigcap_{f\in F_2}A_{y_1f(y_1,y_2)}B_1\right)\neq\emptyset.$$
  We then choose $t_2\in\{1,\dots,r\}$ such that $B_2:=M_{y_2}D_2\cap U_{t_2}\neq\emptyset$.
  The third step of the iteration becomes a little more complicated.
  Using \cref{thm_polvdW} one more time we find $y_3\in\N$ such that
  \begin{multline*}
    D_3:=B_2\cap\left(\bigcap_{f\in F_1}A_{f(y_3)}B_2\cap A_{y_2f(y_2y_3)}B_2\cap A_{y_1y_2f(y_1y_2y_3)}\right) \\
    \cap\left(\bigcap_{f\in F_2}A_{y_2f(y_2,y_3)}B_2\cap A_{y_1y_2f(y_1y_2,y_3)}B_2\cap A_{y_1y_2f(y_1,y_2 y_3)}B_2\right)\\ \cap\left(\bigcap_{f\in F_3}A_{y_1y_2f(y_1,y_2,y_3)}B_2\right)\neq\emptyset.
  \end{multline*}
  We then choose $t_3\in\{1,\dots,r\}$ such that $B_3:=M_{y_3}D_3\cap U_{t_3}\neq\emptyset$.

  In general, for $n\geq2$, assume that $(t_m)_{m=0}^{n-1}$, $(y_m)_{m=1}^{n-1}$, $(B_m)_{m=0}^{n-1}$ and $(D_m)_{m=1}^{n-1}$ have been constructed.
  For each $i\in\{1,\dots,s\}$ and each $f\in F_i$, we define the collection $G_n(f)$ of all functions $g:\Z\to\Z$ of the form
  $$g:z\mapsto y(m_1,n-1)f\big(y(m_1,m_2),\ y(m_2,m_3),\ \dots,\ y(m_i,n-1)\cdot z\big)$$
  for any choice $0\leq m_1<m_2<\cdots<m_i<n$.
  If $i>n$ then we set $G_n(f)$ to be empty.
  Observe that each $g\in G_n(f)$ is a polynomial with rational coefficients satisfying $g(0)=0$.

  Invoking \cref{thm_polvdW}, we can find $y_n\in\N$ satisfying
  \begin{equation}\label{eq_proof_main41}
    D_n:=B_{n-1}\cap\bigcap_{i=1}^s\bigcap_{f\in F_i}\bigcap_{g\in G_n(f)} A_{g(y_n)}B_{n-1}\neq\emptyset.
  \end{equation}
  Let $t_n\in\{1,\dots,r\}$ be such that the intersection $B_n:=M_{y_n}D_n\cap U_{t_n}\neq\emptyset$ (observe that $B_n$ is open because $M_{y_n}$ is an open map).
  This finishes the construction of $y_n$, $t_n$, $D_n$, $B_n$.
  It is immediate from the construction that $B_n\subset U_{t_n}$ for every $n\geq0$.
  Moreover, $B_n\subset M_{y_n}D_n\subset M_{y_n}B_{n-1}$.
  Iterating this observation we obtain
  \begin{equation}\label{eq_proof_maindynamicaltheoremnovo}
  \forall m\leq n,\qquad B_n\subset M_{y(m,n)}B_m.
  \end{equation}

  Since the sequence $(t_n)_{n\geq0}$ takes only finitely many values, there exists $t\in\{1,\dots,r\}$ and infinitely many tuples of natural numbers $n_0<\cdots<n_s$ such that $t_{n_i}=t$.
  For each $i\in\{1,\dots,s\}$, let $x_i=y(n_{i-1},n_i)$.
  We claim that \eqref{eq_proof_thm_maindynamicnovo} is satisfied with $U=U_t$ and with this choice of $x_i$.
  We will show that the intersection in \eqref{eq_proof_thm_maindynamicnovo} is non-empty by proving that it contains $B_{n_s}$.
  Since $B_{n_j}\subset U_t$ for every $j\in\{0,\dots,s\}$, it suffices to show that
  \begin{equation}\label{eq_proof_main_dynamical_last}
    \forall\ 0\leq j<i\leq s,\qquad \forall\ f\in F_{i-j},\qquad B_{n_s}\subset M_{x_{j+1}\cdots x_s} A_{f(x_{j+1},\dots,x_i)}B_{n_j}.
  \end{equation}
  Now fix $0\leq j<i\leq s$ and $f\in F_{i-j}$.
  Observe that there exists some $g\in G_{n_i}(f)$ such that $f(x_{j+1},\dots,x_i)=g(y_{n_i})/y(n_j,n_i-1)$.
  Using \eqref{eq_proof_maindynamicaltheoremnovo}, we conclude
  \begin{eqnarray*}
    B_{n_s}&\subset&M_{y(n_i,n_s)}B_{n_i}\subset M_{y(n_i,n_s)}M_{y_{n_i}}D_{n_i}\\ \text{using \eqref{eq_proof_main41}}&\subset& M_{y(n_i-1,n_s)}\Big(A_{g(y_{n_i})}B_{n_i-1}\Big)\\ \text{using \eqref{eq_proof_maindynamicaltheoremnovo}} &\subset&M_{y(n_i-1,n_s)}A_{g(y_{n_i})}M_{y(n_j,n_i-1)}B_{n_j}\\  \text{using \eqref{eq_distributivity}} &=&M_{y(n_i-1,n_s)}M_{y(n_j,n_i-1)}A_{g(y_{n_i})/y(n_j,n_i-1)}B_{n_j}\\ &=&M_{x_{j+1}\cdots x_s}A_{f(x_{j+1},\dots,x_i)}B_{n_j}.
  \end{eqnarray*}
  This proves \eqref{eq_proof_main_dynamical_last} and finishes the proof of \cref{thm_maindynamical}.

\section{An elementary proof that the \pattern{} $\{x,x+y,xy\}$ is Ramsey}\label{sec_combinatorics}
In this section we present an elementary rendering of the above proof of \cref{thm_mainN}.
To keep things shorter and more elegant, we prove only \cref{cor_{x,xy,x+y}}; the proof in this section can be adapted to obtain the full strength of \cref{thm_mainN}.
We remark that, while this proof is short and essentially self contained, it is, in essence, a combinatorial rephrasing of the dynamical proof.

We will use the following version of van der Waerden's theorem; this version is a particular case of \cite[Theorem 4.5]{Bergelson_Hindman01}.
\begin{theorem}\label{thm_Sz}
  Let $E\subset\N$ be piecewise syndetic, and let $F\subset\N$ be finite. Then there exists $n\in\N$ such that the intersection
  $$S\cap\bigcap_{m\in F}(S-mn)$$
  is piecewise syndetic.
\end{theorem}

\begin{proof}[Proof of {\cref{cor_{x,xy,x+y}}}]
  Let $r\in\N$ and let $\N=C_1\cup\cdots\cup C_r$ be an arbitrary coloring (or partition) of $\N$. We need to find $t\in\{1,\dots,r\}$ and (infinitely many) $x,y\in\N$ satisfying
\begin{equation}\label{eq_2proof_thm_mainN1q}
  \big\{x,x+y,xy\big\}\subset C_t.
\end{equation}
We will construct inductively four sequences:
\begin{itemize}
  \item an increasing sequence $(y_i)_{i\geq1}$ of natural numbers,
  \item two sequences $(B_i)_{i\geq0}$ and $(D_i)_{i\geq1}$ of piecewise syndetic subsets of $\N$,
  \item a sequence $(t_i)_{i\geq0}$ of colors in $\{1,\dots,r\}$,
\end{itemize}
such that $B_i\subset C_{t_i}$ for every $i\geq0$.

Initiate by choosing $t_0\in\{1,\dots,r\}$ such that $C_{t_0}$ is piecewise syndetic, and let $B_0:=C_{t_0}$.
Assume now that $i\geq1$ and that we have already defined $(t_j)_{j=0}^{i-1}$, $(y_j)_{j=1}^{i-1}$, $(B_j)_{j=0}^{i-1}$ and $(D_j)_{j=1}^{i-1}$.
We apply \cref{thm_Sz} to find $y_i\in\N$ such that
\begin{equation}
  \label{eq_proofvdWq}D_i:=B_{i-1}\cap\bigcap_{j=1}^i\Big(B_{i-1}-y_j^2\cdots y_{i-1}^2y_i\Big)
\end{equation}
is piecewise syndetic (with the convention that for $i=j$, the (empty) product $y_j^2\cdots y_{i-1}^2$ equals $1$).
Observe that $y_iD_i$ is also piecewise syndetic, and therefore \cref{prop_pwspartition} provides some $t_i\in\{1,\dots,r\}$ such that $B_i:=y_iD_i\cap C_{t_i}$ is piecewise syndetic.
This finishes the construction of the sequences.

Note that $B_i\subset y_iD_i\subset y_iB_{i-1}$; iterating this fact we obtain
\begin{equation}\label{eq_proof2q}\forall\  0\leq j<i,\qquad B_i\subset y_{j+1}y_{j+2}\cdots y_iB_j.
\end{equation}
Since the sequence $(t_i)$ takes only finitely many values, there exist (infinitely many) $j<i$ such that $t_i=t_j$.
Let $\tilde x\in B_i$, let $y:=y_{j+1}\cdots y_i$, and let $x:=\tilde x/y$.
We claim that $\{x,x+y,xy\}\subset C_{t_i}$, which will complete the proof.
Indeed $xy=\tilde x\in B_i\subset C_{t_i}$ and from \eqref{eq_proof2q} we have $xy\in B_i\subset y B_j$ so $x\in B_j\subset C_{t_j}=C_{t_i}$.
Finally we have
\begin{eqnarray*}
  y(x+y)&=&\tilde x+y^2\in B_i+y^2\subset y_iD_i+y^2\\  \text{using \eqref{eq_proofvdWq}}\qquad&\subset& y_i\big(B_{i-1}-y_{j+1}^2\cdots y_{i-1}^2y_i\big)+y^2 \\ \text{using \eqref{eq_proof2q}}\qquad&\subset&y_i\big(y_{j+1}\cdots y_{i-1}B_j-y_{j+1}^2\cdots y_{i-1}^2y_i\big)+y^2\\&=& yB_j-y^2+y^2=yB_j,
\end{eqnarray*}
which implies that $x+y\in B_j\subset C_{t_j}=C_{t_i}$.
\end{proof}

\begin{remark}
As an alternative approach, one could replace piecewise syndetic sets with sets having positive upper density and replace van der Waerden's theorem with (a suitable form of) Szemer\'edi's theorem in arithmetic progressions \cite{Szemeredi75}.
\end{remark}

\section{Ramsey theoretic applications}\label{sec_corollaries}
In this section we derive some corollaries of our main result, \cref{thm_mainN}, by specifying values of $s$ and sets of functions $F_i$ of interest.
For convenience, we recall the formulation of \cref{thm_mainN}.

\begin{named}{\cref*{thm_mainN}}{}
Let $s\in\N$ and, for each $i=1,\dots,s$, let $F_i$ be a finite set of functions $\N^i\to\Z$ such that for all $f\in F_i$ and any $x_1,\dots,x_{i-1}\in\N$, the function $x\mapsto f(x_1,\dots,x_{i-1},x)$ is polynomial with $0$ constant term.
Then for any finite coloring of $\N$ there exists a color $C\subset\N$ and (infinitely many) $(s+1)$-tuples $x_0,\dots,x_s\in\N$ such that 
$$\{x_0\cdots x_s\}\cup\Big\{x_0\cdots x_j+f(x_{j+1},\dots,x_i):0\leq j<i\leq s, f\in F_{i-j}\Big\}\subset C.$$
\end{named}
By specifying $s=1$ we obtain the following result:
\begin{corollary}\label{cor_onlyy}
  Let $k\in\N$ and let $f_1,\dots,f_k\in\Z[x]$ satisfy $f_\ell(0)=0$ for each $\ell$.
  Then for any finite coloring of $\N$ there exist $x,y\in\N$ such that the set
  $$\{xy,x+f_1(y),\dots,x+f_k(y)\}$$
  is monochromatic.
\end{corollary}
Observe that by putting $f_1(y)=0$, the monochromatic configuration in the previous corollary contains $x$.

In a different direction, letting $s$ be arbitrary but requiring each $F_i$ to consist of only the zero function and the function $f_i(x_1,\dots,x_i)=x_1\cdots x_i$ we deduce:
\begin{corollary}\label{cor_onlyproduct}
  For any $s\in\N$ and any finite coloring of $\N$, there exist $x_0,\dots,x_s\in\N$ such that the set
  $$\left\{\prod_{\ell=0}^jx_\ell\ :\ 0\leq j\leq s\right\}\cup \left\{\prod_{\ell=0}^jx_\ell+\prod_{\ell=j+1}^ix_\ell\ :\ 0\leq j<i\leq s\right\}$$
  is monochromatic.
\end{corollary}

Observe that we do not require that each function $f\in F_i$ in \cref{thm_mainN} be a polynomial in all its variables (but only in the last variable).
In particular, we obtain the following examples:
\begin{example}
  The following are \monochromatic{} \patterns{}:
  \begin{enumerate}
    \item $\{x,x+y,xy,xyz,x+z,x+z^y\}$;
    \item $\{x,xy,xyz,x+f(y)z\}$ for any function $f:\N\to\Z$;
    \item $\{x,xy,xyz,xyzt,x+z^y,x+t^z,x+f(y)t^{g(z)}\}$ for any functions $f,g:\N\to\N$.
  \end{enumerate}
\end{example}
Finally, we prove \cref{cor_partitionregularequationgeneral} from the introduction.
\begin{named}{\cref*{cor_partitionregularequationgeneral}}{}
 Let $k\in\N$ and $c_1,\dots,c_k\in\Z\setminus\{0\}$ be such that $c_1+\cdots+c_k=0$.
  Then for any finite coloring of $\N$ there exist pairwise distinct $a_0,\dots,a_k\in\N$, all of the same color, such that
  \begin{equation}\label{eq_monochromaticcorollary}
  c_1a_1^2+\cdots+c_ka_k^2=a_0.
  \end{equation}
\end{named}
\begin{proof}
Consider the quadratic polynomials
$$p(t)=\sum_{\ell=1}^kc_\ell(1+\ell t)^2,\qquad q(t)=\sum_{\ell=1}^{k-1}c_\ell(1+\ell t)^2+c_k(1+2kt)^2.$$
Both have rational coefficients and a root at $t=0$.
On the other hand, the derivatives
$$p'(t)=2\sum_{\ell=1}^k\ell c_\ell(1+\ell t),\qquad q'(t)=2\sum_{\ell=1}^{k-1}\ell c_\ell(1+\ell t)+4kc_k(1+2kt)$$
can not both vanish at $t=0$.
Therefore at least one of these polynomials must have a second root at some $t\in\Q\setminus\{0\}$.
Assume $p$ has a second root (an analogous argument works in the alternative case).
Letting $d$ be the denominator of $t$ and $u_\ell=d(1+\ell t)$ for each $\ell=1,\dots,k$, we now have pairwise distinct $u_1,\dots,u_k\in\Z$ such that $c_1u_1^2+\cdots+c_ku_k^2=0$.
We can also assume that $c_1u_1+\cdots+c_ku_k\neq0$ by changing some nonzero $u_\ell$ into $-u_\ell$ if necessary.

Let $b=2(c_1u_1+\cdots+c_ku_k)$.
  Let $\chi:\N\to\{1,\dots,r\}$ be an arbitrary finite coloring of $\N$ and define a new coloring $\tilde\chi$ of $\N$ in $r+b-1$ colors by:
  $$\tilde\chi(n):=\left\{\begin{array}{cl}\chi\left(\tfrac nb\right)&\text{ if }n\text{ is divisible by }b\\ r+(n\bmod b)&\text{ otherwise}\end{array}\right.$$
  where $n\bmod b\in\{0,1,\dots,b-1\}$ is the remainder of the division of $n$ by $b$.
  Next apply \cref{cor_onlyy} to find $x,y\in\N$ such that the set $\{x,xy,x+y,x+u_1y,\dots,x+u_ky\}$ is monochromatic with respect to $\tilde\chi$.

  Observe that, in view of the construction of the coloring $\tilde\chi$, all the numbers $x,xy,x+y$ share the same congruence class modulo $b$, which implies that both $x$ and $y$ are divisible by $b$.
  We deduce that the set $\left\{\tfrac{xy}b,\tfrac{x+u_1y}b,\dots,\tfrac{x+u_ky}b\right\}$ consists of integers and is monochromatic with respect to $\chi$.
  Letting $a_0=\tfrac{xy}b$ and $a_\ell=\tfrac{x+u_\ell y}b$ for $\ell=1,\dots,k$, we have the desired relation \eqref{eq_monochromaticcorollary}.
\end{proof}
\section{Extensions to LID}\label{sec_finalremarks}
In this paper so far we have restricted our attention to configurations inside $\N$, but it makes sense to consider analogous questions in a more general setup.
It turns out that our arguments apply without much additional effort to a natural class of rings studied in \cite{Bergelson_Moreira16}, namely the class of LIDs:
\begin{definition}
  An integral domain $R$ is called a \emph{large ideal domain (LID)} if every non-trivial ideal of $R$ has finite index in $R$.
\end{definition}
Examples of LID's include all fields, the ring $\Z$ (and more generally the ring of integers of any number field), and the ring $\F[x]$ of polynomials over a finite field.
Observe that $\N$, not being a ring, is not strictly speaking a LID.
In fact, one can define \emph{LID semirings} (a class which would include $\N$) but we will not pursue this possibility here.

Given a LID $R$, we denote by $\AR$ its \emph{affine semigroup}, defined by $\AR:=\{x\mapsto ax+b:a,b\in R, a\neq0\}$.
The semigroup $\AR$ is a group if and only if $R$ is a field.

The following version of the affine topological correspondence principle for LID can be proved in the same way as \cref{thm_correspondence}.

\begin{theorem}\label{thm_correspondenceLID}
  Let $R$ be a LID and let $\AR$ denote the semigroup of all affine transformations of $R$.
  There exists an $\AR$-topological system $(X,(T_g)_{g\in\AR})$ with a dense set of additively minimal points, such that each map $T_g:X\to X$ is open and injective, and with the property that for any finite coloring $R=C_1\cup\cdots\cup C_r$ there exists an open cover $X=U_1\cup\cdots\cup U_r$ such that for any $g_1,\dots,g_k\in\AR$ and $t\in\{1,\dots,r\}$,
  \begin{equation}\label{eq_thm_correspondenceLID}
    \bigcap_{\ell=1}^kT_{g_\ell}(U_t)\neq\emptyset\qquad\Longrightarrow\qquad\bigcap_{\ell=1}^kg_\ell(C_t)\neq\emptyset.
  \end{equation}
\end{theorem}
The only non-trivial step in generalizing \cref{thm_correspondence} to this setting is the following extension of \cref{prop_pwsaffine}, which crucially relies on the the fact that $R$ is a LID.

\begin{definition}Let $(R,+)$ be an abelian group.
\begin{itemize}
  \item A set $S\subset R$ is called \emph{syndetic} if there exists a finite set $F\subset R$ such that $R=S-F$.
  \item A set $T\subset R$ is called \emph{thick} if for any finite set $F\subset R$ there exists $x\in R$ such that $x+F\subset T$.
  \item A set $B\subset R$ is called \emph{piecewise syndetic} if $B=S\cap T$ for a syndetic set $S\subset R$ and a thick set $T\subset R$.
  \end{itemize}
\end{definition}

\begin{lemma}\label{lemma_pwsLID}
  Let $R$ be a LID and let $B\subset(R,+)$ be piecewise syndetic.
  Then for any $a\in R\setminus\{0\}$, the dilation $aB$ is also piecewise syndetic.
\end{lemma}

\begin{proof}
Let $S$ and $T$ be such that $B=S\cap T$ and $S$ is syndetic and $T$ is thick.
Let $T'=aT\cup(R\setminus aR)$ and let $S'=aS$.
Then clearly $aB=T'\cap S'$.
We now claim that $T'$ is thick and $S'$ is syndetic, which will finish the proof.

Let $F\subset R$ be a finite set such that $S-F=R$.
Then $S'-aF=aR$.
Since $R$ is a LID, the ideal $aR$ has finite index in $R$.
Let $\tilde F$ be a (finite) set of co-set representatives.
Then $aR-\tilde F=R$ and hence $S'-(aF+\tilde F)=R$.
Taking $F':=aF+\tilde F$ we deduce that $S'-F'=R$ and $S'$ is syndetic, as desired.

Next we show that $T'$ is thick.
Let $F\subset R$ be an arbitrary finite set; we will find $x\in R$ such that $x+F\subset T'$.
Split $F=F_1\cup F_2$ where $F_1=F\cap aR$ and $F_2=F\setminus F_1$.
If $F$ is disjoint from $aR$ then it is already contained in $T'$.
Let $F'=F_1/a$ and let $x'\in R$ be such that $x'+F'\subset T$.
Then, taking $x=ax'$ we have $x+F=a(x'+F')\cup ax'+F_2$.
Since $x'+F'\subset T$, the first term $a(x'+F')$ is inside $aT\subset T'$.
Since $F_2$ is disjoint from $aR$, also $ax'+F_2$ is disjoint from $aR$, and hence contained in $T'$.
Therefore $x+F\subset T'$, as desired.
\end{proof}

Observe that \cref{lemma_pwsLID} does not hold in general rings, not even in every principal ideal domain.
An example is provided by the PID $\Q[x]$ of all polynomials with rational coefficients: while $\Q[x]$ is itself a piecewise syndetic set, the ideal $x\Q[x]$ has infinite index as an additive subgroup and hence can not be piecewise syndetic.

One can then obtain a dynamical recurrence result analogous to \cref{thm_maindynamical} which, together with \cref{thm_correspondenceLID}, implies the following combinatorial corollary.

\begin{theorem}\label{thm_main}
Let $R$ be a LID, let $s\in\N$ and, for each $i=1,\dots,s$, let $F_i$ be a finite set of functions $R^i\to R$ such that for all $f\in F_i$ and any $x_1,\dots,x_{i-1}\in R$, the function $x\mapsto f(x_1,\dots,x_{i-1},x)$ is polynomial with $0$ constant term.
Then for any finite coloring of $R$ there exists a color $C\subset R$ and (infinitely many) $(s+1)$-tuples $x_0,\dots,x_s\in R$ such that 
$$\{x_0\cdots x_s\}\cup\Big\{x_0\cdots x_j+f(x_{j+1},\dots,x_i):0\leq j<i\leq s, f\in F_{i-j}\Big\}\subset C.$$
\end{theorem}

The only new ingredient needed to run the proof of \cref{thm_maindynamical} in the LID setting is a suitable version of the polynomial van der Waerden theorem; such a version follows from {\cite[Proposition 7.5]{Bergelson_Leibman99}}.

\bibliography{refs-joel}
\bibliographystyle{plain}
\end{document}